\documentclass[12pt]{article}

\usepackage[pdftex]{color, graphicx}
\usepackage{setspace}
\usepackage[utf8]{inputenc}
\usepackage{amsmath}
\usepackage[pdftex,bookmarks,colorlinks]{hyperref}
\usepackage{latexsym}
\usepackage{amsfonts}
\usepackage{amssymb}
\usepackage{amsthm}
\usepackage{paralist}
\usepackage{mathrsfs}
\begin{document}
\newtheorem{thm}{Theorem}
\newtheorem{cor}[thm]{Corollary}
\newtheorem{lem}{Lemma}
\theoremstyle{remark}\newtheorem{rem}{Remark}
\theoremstyle{definition}\newtheorem{defn}{Definition}

\title{An Extension of Calder\'on Transfer Principle and Its Application to Ergodic Maximal Function}
\author{Sakin Demir}
\maketitle
\begin{abstract}We first prove that the well known transfer principle of A. P. Calder\'on can be extended to the vector-valued setting and then we apply  this extension to vector-valued inequalities for the Hardy-Littlewood maximal function to prove the vector-valued strong type $L^p$ norm inequalities for $1<p<\infty$ and the vector-valued weak type $(1,1)$ inequality for ergodic maximal function.
\end{abstract}
{\bf{Mathematics Subject Classifications:}} 47A35, 28D05, 47A64.\\
{\bf{Key Words:}} Transfer Principle, Ergodic Theory, Vector-valued Inequality.\\
\centerline{}
\indent
 The vector-valued inequalities has long been studied in harmonic analysis, and unfortunately there is not so many results in this direction in ergodic theory. The well known transfer principle of A. P. Calder\'on~\cite{apcal} is a great tool to prove certain type of inequalities in ergodic theory. By using this transfer principle one can use some kind of strong type or weak type inequalities in harmonic analysis to prove some kind of  strong type or weak type inequalities in ergodic theory under certain assumptions. It is, therefore, quite natural to look for a Calder\'on type transfer principle to be able to prove vector-valued inequalities in ergodic theory. In this article we first prove that Calder\'on transfer principle can be extended to the vector-valued setting under the same assumptions and then we apply our result to a theorem of C. Fefferman and E. M. Stein~\cite{fs} to prove the vector-valued strong $L^p$ inequalities for $1<p<\infty$ and the vector-valued weak type $(1,1)$ inequality for the ergodic maximal function.\\
We will assume that $X$ is a measure space which is totally $\sigma$-finite and $U^t$ is a one-parameter group of measure-preserving transformations of $X$. We will also assume that for every measurable function $f$ on $X$ the function $f(U^tx)$ is measurable in the product of $X$ with the real line. $T$ will denote an operator defined on the space of locally integrable functions on the real line with the following properties: the values of $T$ are continuous functions on the real line, $T_j$ is sublinear and commutes with translations, and $T$ is semilocal in the sense that there exists a positive number $\epsilon$ such that the support of $Tf$ is always contained in an $\epsilon$-neighborhood of the support of $f$. \\
We will associate an operator $T^\sharp$ on functions on $X$ with such an operator $T$ as follows:\\
Given a function $f_j$ on $X$ let
$$F_j(t,x)=f_j(U^tx).$$
If $f_j$ is the sum of two functions which are bounded and integrable, respectively, then $F_j(t,x)$ is a locally integrable function of $t$ for almost all $x$ and therefore
$$G_j(t,x)=T(F_j(t,x))$$
is a well-defined continuous function of $t$ for almost all $x$. Thus $g_j(x)=G_j(0,x)$ has a meaning and we define
$$T^\sharp f_j=g_j(x).$$
Let now $(T_n)$ be a sequence of operators as above and define
$$Sf=\sup_n|T_nf|$$
and 
$$S^\sharp f=\sup_n|T_n^\sharp f|.$$
 Suppose that $f=(f_1,f_2,f_3,\dots )$ then we have $Tf=(Tf_1,Tf_2,Tf_3,\dots )$. Let now $1\leq r<\infty$ then 
 $$\|f(x)\|_{l^r}=\left(\sum_{j=1}^\infty |f_j(x)|^r\right)^{1/r}$$
 and similarly
 $$\|Tf(x)\|_{l^r}=\left(\sum_{j=1}^\infty |Tf_j(x)|^r\right)^{1/r}.$$
 Also note that
 $$\|f(x)\|_{l^{\infty}}=\sup_j|f_j(x)|$$
 and
  $$\|Tf(x)\|_{l^{\infty}}=\sup_j|Tf_j(x)|.$$
We can now state and prove our main result.
 \begin{thm}\label{vvctp}Let $1\leq r\leq\infty$ and $1\leq p<\infty$. Suppose that there exists a constant $C_1>0$ such that
 $$\left(\int_{\mathbb{R}}\|Sf(t)\|_{l^r}^p\,dt\right)^{1/p}\leq C_1\left(\int_{\mathbb{R}}\|f(t)\|_{l^r}^p\,dt\right)^{1/p}$$
 where $f=(f_1,f_2,f_3,\dots )$ is a sequence of functions on $\mathbb{R}$.
 Then we have
  $$\left(\int_X\|S^\sharp f(x)\|_{l^r}^p\,dx\right)^{1/p}\leq C_1\left(\int_X\|f(x)\|_{l^r}^p\,dx\right)^{1/p}$$
  where $f=(f_1,f_2,f_3,\dots )$ is a sequence of functions on $X$.
 Suppose that there exists a constant $C_2>0$ such that for all $\lambda >0$
 $$\left|\left\{t\in\mathbb{R}:\|Sf(t)\|_{l^r}>\lambda\right\}\right|\leq\frac{C_2^p}{\lambda^p}\int_{\mathbb{R}}\|f(t)\|_{l^r}^p\,dt$$
 where $f=(f_1,f_2,f_3,\dots )$ is a sequence of functions on $\mathbb{R}$.
 Then for all $\lambda >0$ we have
 $$\left|\left\{x\in X:\|S^\sharp f(x)\|_{l^r}>\lambda\right\}\right|\leq\frac{C_2^p}{\lambda^p}\int_X\|f(x)\|_{l^r}^p\,dx$$
 where $f=(f_1,f_2,f_3,\dots )$ is a sequence of functions on $X$.
\end{thm}
 \begin{proof}We apply the argument of A. P. Calder\'on~\cite{apcal} to the vector-valued setting to prove our theorem.\\
 Without loss of generality we may assume that the sequence $T_n$ is finite, for if the theorem is established in this case, the general case follows by a passage to the limit. Under this assumption the operator $S$ has the same properties as the operator $T$ above. Let
$$F(t,x)=(F_1(t,x),F_2(t,x),F_3(t,x),\dots )$$
and
$$G(t,x)=(G_1(t,x),G_2(t,x),G_3(t,x),\dots ).$$ 
 As we did before, we let 
$$G_j(t,x)=S(F_j(t,x))$$
for all $j=1,2,3,\dots$ .
We note that 
$$F_j(t,U^sx)=F_j(t+s,x),$$
which means that for any two given values $t_1$, $t_2$ of $t$, $F_j(t_1,x)$ and $F_j(t_2,x)$ are equimeasurable functions of $x$. On the other hand, due to translation invariance of $S$, the function $G_j(t,x)$ has the same property. In fact we have
$$G_j(t,U^sx)=S(F_j(t,U^sx))=S(F_j(t+s,x))=G_j(t+s,x).$$
Let now $F_j^a(t,x)=F_j(t,x)$ if $|t|<a$, $F_j^a(t,x)=0$ otherwise, and let
$$G_j^a(t,x)=S(F_j^a(t,x)).$$
Since $S$ is positive (i.e., its values are nonnegative functions) and sublinear, we have
\begin{align*}
G_j(t,x)=S(F_j)&=S(F_j^{a+\epsilon}+(F_j-F_j^{a+\epsilon}))\\
& \leq S(F_j^{a+\epsilon})+S(F_j-F_j^{a+\epsilon})
\end{align*}
and since $F_j-F_j^{a+\epsilon}$ has support in $|t|>a+\epsilon$, and $S$ is semilocal, the last term on the right vanishes for $|t|\leq a$ for $\epsilon$ sufficiently large, independently of $a$. Thus we have $G_j\leq G_j^{a+\epsilon}$ for $|t|\leq a$. Suppose that $S$ satisfies a vector-valued strong $L^p$ norm inequality. Then since $G_j(0,x)$ and $G_j(t,x)$ are equimeasurable functions of $x$ for all $j=1,2,3\dots$, we have
\begin{align*}
2\int_X\|G(0,x)\|_{l^r}^p\,dx&=\frac{1}{a}\int_{|t|<a}\,dt\int_X\|G(t,x)\|_{l^r}^p\,dx\\
&\leq\frac{1}{a}\int_{|t|<a}\,dt\int_X\|G^{a+\epsilon}(t,x)\|_{l^r}^p\,dx\\
&=\frac{1}{a}\int_X\,dx\int_{|t|<a}\|G^{a+\epsilon}(t,x)\|_{l^r}^p\,dt
\end{align*}
and since $SF_j^{a+\epsilon}=G_j^{a+\epsilon}$
$$\int_{|t|<a}\|G^{a+\epsilon}(t,x)\|_{l^r}^p\,dt\leq C_1^p\int\|F^{a+\epsilon}(t,x)\|_{l^r}^p\,dt,$$
whence substituting above we obtain
$$2\int_X\|G(0,x)\|_{l^r}^p\,dx\leq\frac{1}{a}C_1^p\int_X\,dx\int\|F^{a+\epsilon}(t,x)\|_{l^r}^p\,dt$$
and again, since $F_j(0,x)$ and $F_j(t,x)$ are equimeasurable for all $j=1,2,3\dots$, the last integral is equal to
$$2(a+\epsilon)\int_X\|F(0,x)\|_{l^r}^p\,dx$$
and
$$\int_X\|G(0,x)\|_{l^r}^p\,dx\leq\frac{1}{a}(a+\epsilon)C_1^p\int_X\|F(0,x)\|_{l^r}^p\,dx.$$
Letting $a$ tend to infinity, we prove the first part of our theorem.\\
    Suppose now that $S$ satisfies a vector-valued weak type $(p,p)$ inequality. For any given $\lambda >0$, let $E$ and $\widetilde{E}$ be the set of points where $\|G(0,x)\|_{l^r}>\lambda$ and $\|G^{a+\epsilon}(t,x)\|_{l^r}>\lambda$, respectively, and $\widetilde{E}_y$ the intersection of $\widetilde{E}$ with the set $\{(t,x):x=y\}$. Then we have
$$2a|E|\leq |\widetilde{E}|=\int_X|\widetilde{E}_x|\,dx.$$
On the other hand, since $S$ satisfies a vector-valued weak type $(p,p)$ inequality, we have
$$|\widetilde{E}_x|\leq\frac{C_2^p}{\lambda^p}\int \|F^{a+\epsilon}(t,x)\|_{l^r}^p\,dt.$$
By using the above inequalities and the fact that $F_j(0,x)$ and $F_j(t,x)$ are equimeasurable for all $j=1,2,3,\dots$ we have
$$a|E|\leq\frac{C_2^p}{\lambda^p}(a+\epsilon )\int \|F(0,x)\|_{l^r}^p\,dx.$$
When we let $a$ tend to $\infty$ we find the desired result.
\end{proof}
Let $f$ be a function on $\mathbb{R}$, the Hardy-Littlewood maximal function $Mf$ is defined by
$$Mf(t)=\sup_{t\in I}\frac{1}{|I|}\int_I|f(y)|\,dy$$
where $I$ denotes an arbitrary interval in $\mathbb{R}$.
Consider now an ergodic measure preserving transformation $\tau$ acting on a probability space $(X,\beta,\mu )$  then $M^\sharp$ is the ergodic maximal function defined by
$$M^\sharp f(x)=\sup_n\frac{1}{n}\sum_{k=0}^{n-1}\left|f(\tau^kx)\right|$$
where $f$ is a function on $X$.\\
Recall the following result of C. Fefferman and E. M. Stein~\cite{fs}.
\begin{lem}\label{fslem}Let $f=(f_1,f_2,f_3,\dots )$ be a sequence of functions on $\mathbb{R}$ and $1<r, p<\infty$. Then there are constants $C_{r,p}>0$ and $C_r>0$ such that
\begin{enumerate}[(a)]
\item $\left\|\left(\sum_{j=1}^\infty |Mf_j(\cdot )|^r\right)^{1/r}\right\|_p\leq C_{r,p}\left\|\left(\sum_{j=1}^\infty |f_j(\cdot )|^r\right)^{1/r}\right\|_p$,
\item $\left|\left\{t\in\mathbb{R}:\left(\sum_{j=1}^\infty |Mf_j(t)|^r\right)^{1/r}>\lambda\right\}\right|\leq\frac{C_r}{\lambda}\left\|\left(\sum_{j=1}^\infty |f_j(\cdot )|^r\right)^{1/r}\right\|_1$\\
 for all $\lambda >0$.
\end{enumerate}
\end{lem}
When we apply Theorem~\ref{vvctp} to Lemma~\ref{fslem} we find the following result.
\begin{thm}\label{vvifemf}Let $f=(f_1,f_2,f_3,\dots )$ be a sequence of functions on $X$ and $1<r, p<\infty$. Then there are constants $C_{r,p}>0$ and $C_r>0$ such that
\begin{enumerate}[(a)]
\item $\left\|\left(\sum_{j=1}^\infty |M^\sharp f_j(\cdot )|^r\right)^{1/r}\right\|_p\leq C_{r,p}\left\|\left(\sum_{j=1}^\infty |f_j(\cdot )|^r\right)^{1/r}\right\|_p$,
\item $\mu\left\{x\in X:\left(\sum_{j=1}^\infty |M^\sharp f_j(x)|^r\right)^{1/r}>\lambda\right\}\leq\frac{C_r}{\lambda}\left\|\left(\sum_{j=1}^\infty |f_j(\cdot )|^r\right)^{1/r}\right\|_1$\\
 for all $\lambda >0$.
\end{enumerate}
\end{thm}
Note that the constants $C_{r,p}$ and $C_r$ in Theorem~\ref{vvifemf} are the same constants as in Lemma~\ref{fslem}.

\vspace{1 cm}
\noindent
Sakin Demir\\
E-mail: sakin.demir@gmail.com\\
Address:\\
İbrahim Çeçen University\\
Faculty of Education\\
04100 Ağrı, TURKEY.
\end{document}